\documentclass[a4paper,10pt,reqno]{amsart}
\usepackage[utf8]{inputenc}

\usepackage[all,cmtip]{xy}
\usepackage{color}
\usepackage{leftidx}

\usepackage{amssymb}
\newtheorem{proposition}{Proposition}
\newtheorem{definition}[proposition]{Definition}

\newtheorem{theorem}[proposition]{Theorem}
\newtheorem{corollary}[proposition]{Corollary}

\newcommand{\C}{\mathbb{C}}

\DeclareMathOperator{\glob}{glob}

\DeclareMathOperator{\Val}{Val}
\DeclareMathOperator{\V}{\mathcal V}

\newcommand{\CPnlam}{\C \mathrm P_\lambda^n}
\DeclareMathOperator{\CP}{\C P}
\newcommand{\CPn}{\CP^n}

\DeclareMathOperator{\Sym}{\operatorname {Sym}}

\DeclareMathOperator{\id}{id}
\newcommand\vol{\operatorname{vol}}

\newcommand \valun{\Val^{U(n)}}

\newcommand{\curvinfty}{\mathcal {C}}
\newcommand{\Cn}{\mathbb{C}^n}

\newcommand{\curv}{\mathrm{Curv}}
\newcommand{\curvun}{{\mathrm{Curv}}^{U(n)}}
\newcommand{\Area}{\mathrm{Area}}
\newcommand{\areaun}{{\mathrm{Area}}^{U(n)}}
\newcommand{\val}{\mathrm{Val}}

\newcommand{\R}{\mathbb{R}}

\newcommand{\algint}[2]{#1 \cdot #2}
\newcommand{\cf}{\kappa}
\newcommand{\cm}{\boldsymbol{b}}

\title{Contact measures in isotropic spaces}
\author{Gil Solanes}
\email{solanes@mat.uab.cat}
\address{Departament de Matem\`atiques\\ Universitat Aut\`onoma de Barcelona\\08193 Bellaterra\\Catalonia\\ Spain.}

\thanks{The author is a Serra H\'unter Fellow and was partially supported by FEDER-MINECO grant MTM2012-34834.}

\begin{document}
\begin{abstract}
We revisit the contact measures introduced by Firey, and further developed by Schneider and Teufel, from the perspective of the theory of valuations on manifolds. This reveals a link between the kinematic formulas for area measures studied by Wannerer and the integral geometry of curved isotropic spaces. As an application we find explicitly the kinematic formula for the surface area measure in Hermitian space.
\end{abstract}
\maketitle

\section{Introduction}
The principal kinematic formula of Blaschke, Santaló and Chern states
\begin{equation}\label{pkf}
 \int_{\overline{SO(n)}}\chi(A\cap gB)dg=\omega_{n}^{-1}\sum_{i=0}^n{n\choose i}^{-1}\omega_i\omega_{n-i}\mu_i(A)\mu_{n-i}(B),
\end{equation}
where  $\chi$ is the Euler characteristic, and $A,B\subset\mathbb R^n$ are sufficiently nice  compact subsets of $\mathbb R^n$ (e.g. convex bodies or smooth submanifolds). The integral over the group of rigid motions $\overline{SO(n)}=SO(n)\ltimes\mathbb R^n$ is performed with respect to the Haar measure $dg$, suitably normalized. On the right hand side, the constants $\omega_i$ denote the volume of the unit $i$-dimensional ball, and the functionals  $\mu_i$ are the so-called {\em intrinsic volumes} (also known as  quermassintegrals or Lipschitz-Killing curvature integrals). On the space of convex bodies $\mathcal K^n$, intrinsic volumes  are {\em valuations}:  a (real-valued) valuation is a functional $\varphi:\mathcal K^n\to\R$ such that
\[
 \varphi(K\cup L)=\varphi(K)+\varphi(L)-\varphi(K\cap L),
\]whenever $K,L,K\cap L\in \mathcal K^n$.
Hadwiger's characterization theorem states that the space of continuous $\overline{SO(n)}$-invariant valuations is spanned by $\mu_0,\ldots, \mu_n$. In particular, this fundamental result yields a simple proof of the principal kinematic formula for convex bodies.

Recent results by S. Alesker (cf. e.g. \cite{alesker_rotation,alesker_solution,alesker_hard}) have allowed important progress in the theory of valuations. This includes, for instance, classification results and kinematic formulas for tensor-valued valuations (cf. e.g. \cite{ bernig_hug, hug_schneider_schuster_a, hug_schneider_schuster_b}) and also for valuations taking values on the space of convex bodies (cf. e.g. \cite{abardia12, abardia_bernig, haberl10, ludwig_2005, schuster10, schuster_wannerer}). 

Another line of research is the determination of kinematic formulas with respect to different groups.
Indeed, Alesker has shown that characterization theorems in the style of Hadwiger's exist also when the group $SO(n)$ is replaced by any compact group $H$ acting transitively on the unit sphere. More precisely it was proved  in \cite{alesker_advances00} that the space $\Val^H$ of continuous, translation-invariant and $H$-invariant valuations has finite dimension. The connected groups $H$ satisfying the previous condition were classified in \cite{borel,montgomery_samelson}. For some of them, namely $H=SO(n),U(n), SU(n), Sp(2)Sp(1), G_2,Spin(7)$, a basis of the space $\Val^H$ has been constructed (cf. \cite{alesker_hard, bernig_sun,bernig_exceptional, bernig_solanes_h2}). For the rest of groups, namely $H=Sp(n),Sp(n)U(1),Sp(n)Sp(1),Spin(9)$, only the dimension of $\Val^H$ is known  (cf. \cite{bernig_hn,bernig_voide}). Let us mention that the case of non-compact isotropy groups is also interesting and has been studied (cf. \cite{alesker_faifman, bernig_faifman,ludwig_reitzner}).

As in the classical case, the fact that $\Val^H$ has finite dimension for any compact group $H$ acting transitively on the unit sphere implies the existence of kinematic formulas in the style of \eqref{pkf} with respect to $H$. The actual computation of these formulas is a difficult problem, which has been recently solved in all cases where a basis of $\Val^H$ is known \cite{bernig_sun,bernig_exceptional,bernig_hn,hig,bernig_solanes_kin}. This was possible thanks to an algebraic approach developed by Bernig and Fu \cite{hig} and based on the product of valuations discovered by Alesker \cite{alesker_product}.

\bigskip More recently, Alesker developed a theory of valuations on smooth manifolds  (cf. \cite{ale} and the references therein). By a previous result of Fu \cite{fu.indiana}, it turns out that kinematic formulas, expressible in terms of such valuations, exist in any Riemannian isotropic space. Such a space is a Riemannian manifold $M$ together with a group  $G$ acting on $M$ by isometries, and such that the induced action on the sphere bundle $SM$ is transitive. The precise statement is the following: given a basis $\varphi_1,\ldots,\varphi_d$ of the space  $\mathcal V(M)^G$ of $G$-invariant valuations on $M$ (cf. Definition \ref{defval}), there exists $$k(\chi)=\sum_{ij} c_{ij}\varphi_i\otimes\varphi_j\in \mathcal V(M)^G\otimes \mathcal V(M)^G $$ such that
\[
\int_G \chi(A\cap gB)dg=\sum_{i,j} c_{ij}\varphi_i(A)\varphi_j(B)= k(\chi)(A,B).
\]
In fact there exist analogous formulas $k(\varphi)$ for any $G$-invariant valuation $\varphi$, which yields a map $k:\mathcal V^G(M)\rightarrow \mathcal V^G(M)\otimes \mathcal V^G(M)$. Furthermore, there are local versions of the notion of valuation called {\em curvature measures}, which assign, to each sufficiently nice compact set,   a signed Borel measure supported on that set (cf. Definition \ref{defval}). Also at this local level, kinematic formulas exist in any Riemannian isotropic space by the results of \cite{fu.indiana}. These local kinematic formulas can be encoded by a linear map $K:\mathcal C(M)^G\rightarrow \mathcal C(M)^G\otimes \mathcal C(M)^G$, where $\mathcal C(M)^G$ stands for the space of $G$-invariant curvature measures of $M$.

The complete list of Riemannian isotropic spaces $(M,G)$ was obtained by Tits \cite[\S 4D]{tits}. If $G$ is connected, it is the following: 
first one has $M=V$,  a Euclidean affine space, and $G=H\ltimes V$, where $H\leq SO(V) $ is a subgroup acting transitively on the unit sphere $S(V)$. In the next cases, $M$ is a rank one symmetric space, and $G$ is the identity component of the isometry group. Finally, there are some exceptional spaces, namely $M=\mathrm S^6$ or $\mathbb R\mathrm P^6$ with $G=G_2$, and $M=\mathrm S^7$ or $\mathbb R\mathrm P^7$ with $G= Spin(7)$.
Except for the classical case of the real space forms $M=\mathbb R\mathrm P^n, \mathrm S^n,\mathrm H^n$, the determination of the principal kinematic formulas in curved isotropic spaces is a challenge. Only very recently it has been possible to compute $k(\chi)$ in the complex space forms $\CPn,\C\mathrm H^n$ (cf. \cite{bfs}). The approach was again algebraic and made use of a module structure on the space of curvature measures over the algebra of valuations. Another important tool was the so-called transfer principle which roughly states that the local kinematic formulas are the same on all spaces with the same isotropy group.

\bigskip
In the flat case, there is a second type of kinematic formula which involves the Minkowski addition instead of the intersection of sets. In this case, integration is performed over the isotropy group. For instance, for $K,L$ convex bodies in Euclidean space
\[
\int_{SO(n)}\vol(K+hL)dh=\omega_{n}^{-1}\sum_{i=0}^n{n\choose i}^{-1}\omega_i\omega_{n-i}\mu_i(K)\mu_{n-i}(L).
\]

These additive kinematic formulas admit a localization in terms of {\em area measures}, which attach a signed measure on the sphere to each convex body (cf. Definition \ref{defarea}). This was proved by Schneider \cite{schneider} for the group $SO(n)$, and more recently by Wannerer \cite{wannerer.aim} for any compact group $H$ acting transitively on the unit sphere $\mathrm S^{n-1}$.  For instance, for the classical surface area measure $S_{n-1}$, there exist constants $c_{ij}$ such that, for every pair $K,L$ of  convex bodies, and for any Borel domains $U,V\subset \mathrm S^{n-1}$, one has
\[
 A(S_{n-1})(K,U,L,V):=\int_{H} S_{n-1}(K+hL,U\cap hV)dh=\sum_{i,j}c_{ij}\Phi_i(K,U)\Phi_j(L,V),
\]
where $\Phi_0,\ldots, \Phi_d$ is a basis of the space $\mathrm{Area}^H$ of $H$-invariant area measures; i.e. $A(S_{n-1})\in \mathrm{Area}^H\otimes \mathrm{Area}^H$. In the classical case $H=SO(n)$ this formula was explicitly obtained by Schneider \cite{schneider}. In the unitary case $H=U(n)$ all local additive kinematic formulas, encoded by a map $A\colon \mathrm{Area}^H\rightarrow \mathrm{Area}^H\otimes \mathrm{Area}^H$, have been obtained by Wannerer  (cf. \cite{wannerer.aim}). This difficult task was accomplished thanks to a new approach based on tensor valued valuations.  For the other groups acting transitively on spheres these formulas are still to be found.

\bigskip
Our main result is the following relation between the principal kinematic formula $k(\chi)$ in a curved isotropic space $(M,G)$ and the local additive kinematic formula $A(S_{n-1})$ in the affine isotropic space  with the same isotropy group. The precise definition of the maps $\glob, \delta, \rho$ will be given in Section 2.

\begin{theorem}\label{mainthm}Let $(M,G)$ be a Riemannian isotropic space with isotropy group $H$, and let $(F,H\ltimes F)$ denote an affine isotropic space with the same isotropy group. Let $\delta:\V(M)\rightarrow\curvinfty(M)$ be the first variation map introduced in \cite{hig}, and let $\glob:\curvinfty(M)\rightarrow \V(M)$ be the globalization map obtained by evaluation of measures on the whole space. There is a natural inclusion $\rho\colon \mathrm{Area}^H\rightarrow \mathcal C(M)^G$ of the space of invariant area measures of $F$ into the space of invariant curvature measures of $M$, such that
\begin{equation}\label{main}
(\delta\otimes \id) k(\chi)=(\rho\otimes(\glob\circ \rho\circ r)) A(S_{n-1}),
\end{equation}
where $r$ is the involution of $\mathrm{Area}^H$ given by $r(\Phi)(K,U)=\Phi(-K,-U)$.
\end{theorem}

We believe that equation \eqref{main} will be useful in the determination of kinematic formulas of curved spaces. Indeed, after \cite{hig} and \cite{wannerer.aim} there are many tools available for the computation of kinematic formulas in affine spaces. Hence, equation \eqref{main} could in principle be used to solve for $k(\chi)$ after computing $A(S_{n-1})$. 

On the other hand, one can also use \eqref{main} in the opposite direction; i.e. to deduce $A(S_{n-1})$ from the knowledge of $k(\chi)$. Indeed, in Section \ref{Hermitian} we use the results of \cite{bfs} to compute the local additive kinematic formula for the surface area measure in the complex space $\C^n$ under the action of $U(n)$ (see Theorem \ref{app}). An algorithm to find this formula (among many others) was already given in \cite{wannerer.aim}, but our results are more explicit.  Another approach to these formulas, giving also explicit results, has been recently developed by Bernig in \cite{bernig_dual}.

The proof of Theorem \ref{mainthm} is based on the notion of contact measure, which was introduced by Firey  \cite{firey} for convex sets in Euclidean space, and was extended by Teufel to homogeneous spaces   \cite{teufel}. The geometric idea behind this notion is to measure the set of positions in which a moving body touches another one tangentially. Here we give a new construction of contact measures which is equivalent to the classical ones, but is better adapted to the language of valuations. Then, we show a transfer principle which relates the contact measures of different isotropic spaces with the same isotropy group. Together with a classical result of Schneider on contact measures in flat spaces (cf. \cite{schneider}), this yields our main result.

\medskip\noindent\textbf{Acknowledgements.} It is a pleasure to thank J. Abardia, A. Bernig, M. Saienko, E. Teufel and T. Wannerer for useful comments on an earlier version of this paper.

\section{Background}
\subsection{Valuations, curvature measures and area measures}
It is known that kinematic formulas apply to many different classes of sets, including sets of positive reach, subanalytic sets, or the recently studied WDC sets (cf. \cite{wdc}). For simplicity we will work here with the class of {\em simple smooth polyhedra} also called {\em submanifolds with corners} (see \cite[Definition 3.1]{fu.intersection}). We  denote by $\mathcal P(M)$ the class of compact simple smooth polyhedra in a manifold $M$. Each  $A\in \mathcal P(M)$ admits a normal cycle $N(A)$ which consists of the set of outer normal vectors of $A$. As a subset of the sphere bundle $SM$ of unit tangent vectors to $M$, the normal cycle is a finite union of relatively compact smooth submanifolds and defines a closed Legendrian current. 

\begin{definition}\label{defval}
 A {\em (smooth) valuation}  on a Riemannian manifold $M$ is a functional $\mu\colon \mathcal P(M)\rightarrow \R$ of the form
\[
 \mu(A)=\int_{N(A)} \omega +\int_A \eta
\]
where $\omega\in\Omega^{n-1}(SM), \eta\in \Omega^n(M)$  are fixed differential forms. We denote by $\V(M)$ the space of valuations of $M$. We also denote by $\val(F)$ the space of smooth translation-invariant valuations on an affine space $F$.

 A {\em (smooth) curvature measure}   on $M$ is a functional that assigns, to each $A\in\mathcal P(M)$ a signed Borel measure $\Phi(A,\cdot)$ of the form
\[
 \Phi(A,U)=\int_{N(A)\cap \pi^{-1}U} \omega +\int_{A\cap U} \eta,\qquad U\subset M,
\]
where $\omega\in\Omega^{n-1}(SM), \eta\in \Omega^n(M)$  are fixed, and $\pi\colon SM\to M$ is the bundle projection. We denote by $\curvinfty(M)$ the space of curvature measures of $M$. We also denote by $\curv(F)$ or just by $\curv$ the space of translation-invariant curvature measures on an affine space $F$.
\end{definition}
There is an obvious map $\glob:\mathcal V(M)\rightarrow \mathcal C(M)$ called {\em globalization map}, which is given by $\glob(\Phi)(A)=\Phi(A,M)$.

\medskip When the ambient manifold is  a Euclidean affine space $F$,  the sphere bundle $SF$ is the cartesian product  of $F$ with the unit sphere $S(F)$. Then, it makes sense to localize valuations using area measures instead of curvature measures. In this setting it is natural to assume translation invariance, and to consider the space $\mathcal K(F)$ of convex bodies instead of $\mathcal P(F)$.
\begin{definition}\label{defarea} An {\em area  measure}  on an $n$-dimensional Euclidean affine space  $F$ assigns, to each $K\in\mathcal K(F)$, a signed Borel measure $\Psi(K,\cdot)$ on the unit sphere $S(F)$ of the form
\[
 \Psi(K,U)=\int_{N(K)\cap \pi_S^{-1}U} \omega ,\qquad U\subset S(F), 
\]
where $\omega\in\Omega^{n-1}(SF)$ is a fixed translation-invariant differential form, and $\pi_S$ is the projection of $SF=F\times S(F)$ on the second factor. We denote by $\Area(F)$ or just by $\Area$  the space of area measures of $F$.
\end{definition}

A classical example of   an area measure is the so-called {\em surface area measure} $S_{n-1}$ which corresponds to $\omega\in \Omega^{n-1}(SF)$ given by
$$\omega_{\xi}=\iota_{\pi_S(\xi)}\pi^*(d\vol),\qquad \xi\in SF,$$ 
where $d\vol$ stands for the volume element of $F$.

There is a well defined injection $i\colon \Area(F)\rightarrow \curv(F)$ such that, for any $K\in\mathcal K(F)$ with smooth positively curved boundary,
\[
 i(\Psi)(K,U)=\Psi(K,\gamma_K(U\cap \partial K)),\qquad U\subset F,
\]
where $\gamma_K:\partial K\rightarrow S(F)$ is the Gauss map of $\partial K$ (cf. \cite{wannerer.aim}). Moreover, if $H$ is a linear group acting transitively on $S(F)$, and $G=H\ltimes F$, then $i\colon \Area^H\rightarrow \curv(F)^G$ and has cokernel of dimension one.

\subsection{Transfer map}Let us recall the transfer map introduced in \cite{bfs}. We consider specifically the case of $G$-invariant curvature measures in an isotropic space $(M,G)$.
Fix $\xi\in SM$, and let $K$ be the subgroup of elements of $G$ fixing $\xi$ (i.e. the isotropy group of $\xi$ with respect to the action of $G$ on $SM$). Let $x=\pi(\xi)\in M$ and denote by $H$  the isotropy group of $x$. Putting $F=T_xM$, one can identify the $G$-invariant differential forms of $SM$ with the $H\ltimes F$-invariant differential forms of $SF$ through
\begin{equation}\label{identif_forms}
 \Omega^*(SM)^G\cong (\Lambda^* T_\xi^* SM)^K\cong \left(\Lambda^* (F^*\oplus (\xi^\bot)^*)\right)^K\cong \Omega^*(SF)^{H\ltimes F}.
\end{equation}
In the middle step we used the decomposition 
\begin{equation}\label{decomp}
 T_\xi SM=T_{\pi(\xi)}M\oplus \xi^\bot
\end{equation}
induced by the Levi-Civita connection of $M$. As shown in \cite[Proposition 2.5]{bfs}, the identification \eqref{identif_forms} induces a linear isomorphism 
\[
 \tau\colon \curv(F)^H\longrightarrow \mathcal{C}(M)^G
\]
from the space of  $H$-invariant, translation-invariant curvature measures of $F$ to the space of $G$-invariant curvature measures of $M$. The inclusion $\rho\colon \Area^H\to\mathcal C(M)$ in Theorem \ref{mainthm} is simply defined by $\rho=\tau\circ i$.

\subsection{First variation}
The first variation of valuations in a Riemannian manifold $M$  was described in \cite{hig} by a map
\[
 \delta\colon \mathcal{V}(M)\rightarrow \mathcal{C}(M)
\]
such that, for any smooth regular domain $A\subset M$ and any smooth $1$-paramater family of diffeomorphisms $\{\phi_t\}$ of $M$ one has
\begin{equation}\label{delta_description}
 \frac{d}{dt}\mu(\phi_t(A))=\int_{\partial A} \left\langle \nu_A(x),\frac{\partial\phi_t}{\partial t}(x)\right\rangle\delta\mu(A,dx),
\end{equation}
where $\nu_A$ denotes the outward pointing unit normal vector field to $\partial A$.  The differential forms defining the curvature measure $\delta \mu$ can be obtained from those definig the valuation $\mu$ by means of the so-called Rumin differential (cf.  \cite{bb,rumin}).

For an affine space $F$, an area measure valued version of the first variation map was introduced in \cite{wannerer.jdg} as the map
\[
 \delta_A\colon \Val(F)\rightarrow \Area(F)
\]
such that $\delta|_{\Val}=i\circ\delta_A$.

\section{Contact measures}
Contact measures were introduced by Firey \cite{firey} for convex sets in Euclidean space, and by Teufel \cite{teufel} for submanifolds in homogeneous spaces. Our construction is essentially equivalent to those, but is more adapted to the language of valuations that has been   surveyed above. The geometric idea remains the same: to measure the set of positions in which a moving body touches another one.

Let $(M,G)$  be a Riemannian isotropic space. This means that $G$ is a connected Lie group acting isometrically on the Riemannian manifold $M$, and such that the induced action on $SM$ is transitive.

Consider ${E=SM\times G}$ and the double fibration
\begin{displaymath}
 \xymatrix{
&  E \ar[dl]_{p} \ar[dr]^{q} &  \\
SM\times SM & & G}
\end{displaymath}
given by $p({\eta,g})=(g\eta,\eta)$ and $q({\eta,g})=g$. Note that $q(p^{-1}(\xi,\eta))$ is the set of $g\in G$ such that $g\eta=\xi$.

In the following we denote by $\pi$ the bundle projection of $SM$ onto $M$, and $\pi_1,\pi_2\colon SM\times SM\to SM$ will be the projections onto each factor.

We normalize the Haar measure of $G$ as in \cite[(2.21)]{bfs}. We also fix an orientation on $G$ and let $dg$ be the volume form associated to the Haar measure with the chosen orientation. Define
\begin{equation}\label{cf}
\cf= -\iota_{T_1}p_*q^*dg\in\Omega^{2n-2}(SM\times SM),
\end{equation}
where $T_1$ is the  vector field in $SM\times SM$ given by $T_1(\xi,\eta)=(T(\xi),0)$,  and $T$ is the vector field on $SM$ which generates the geodesic flow in $SM$. We are taking on $SM$ the same orientation as in \cite[\S 3.1.1]{fu.intersection} and the product orientation wherever it applies. For push-forwards we follow the same convention as in  \cite[\S 2]{fu.intersection}. For the preimage of a submanifold by a submersion, we take the orientation specified in \cite[\S 2.1.3]{fu.intersection}.

Let $s$ be the involution $\xi\mapsto -\xi$ in $SM$. Given $B\in\mathcal P(M)$ we will consider both $N(B)$ and $s(N(B))$. On $N(B)$ we take the orientation described in \cite[\S 3.2]{fu.intersection}, and we orientate $s(N(B))$ so that $s|_{N(B)}$ is orientation reversing. If $B$ is a regular domain with smooth boundary $\partial B=C$, and $N(B)\cup s(N(B))$ is naturally identified with the boundary of a tubular neigborhood $C_\epsilon$ around $C$, then the orientations we have taken correspond to the one induced by $C_\epsilon$ on its boundary.

\begin{definition}
The {\em contact measure} of an isotropic space $(M,G)$ is the element $\cm\in \mathcal{C}(M)^G\otimes \mathcal{C}(M)^G$ given by
\[
 \cm(A,U,B,V)= 
 \int_{(N(A)\cap \pi^{-1}U)\times s(N(B)\cap \pi^{-1}V)} \cf,
\]where $A,B\in\mathcal P(M)$ and $U,V\subset M$ are Borel sets. 
\end{definition}

The geometric meaning of contact measures is the following: $\cm(A,U,B,V)$ measures the set of positions in which a copy of $B$ moving under $G$ touches $A$ in such a way that the contact occurs at points of $U,V$ respectively. This interpretation is already implicit in the previous construction but will be more clear after Corollaries \ref{meaning} and  \ref{firey}. The negative sign in \eqref{cf} is taken so that contact measures on pairs of locally convex sets are positive (cf. Corollary \ref{thm_smooth}).

\bigskip Recall that $SM$ is endowed with a canonical contact form $\alpha$ given by $\alpha_\xi=\langle d\pi,\xi\rangle$. A submanifold $P$ of $SM$ is called {\em Legendrian} if $\alpha|_P\equiv 0$. A vector field $Y$ in $SM$ is called a {\em contact vector field} if $\mathcal L_Y\alpha=f\alpha$ for some function $f$. For instance, the vector field $T$ associated to the geodesic flow is a contact vector field as
\[
 \alpha(T)=1, \quad \mathcal L_T \alpha=0.
\]
Moreover, these two conditions characterize $T$ as the {\em Reeb vector field} corresponding to $\alpha$. Note also that the image of a Legendrian submanifold under the flow associated to a contact vector field is also Legendrian. Given a vector field $X$ in $M$ there is a unique contact vector field $X^c$ in $SM$, called {\em complete lift} of $X$, such that $d\pi X^c=X$ (cf. \cite{yano.ishihara}).

In the following we denote by $A\cdot B$ the algebraic intersection of two oriented submanifolds $A,B$ of complementary dimension intersecting transversely (cf. e.g. \cite{guillemin.pollack}). 

\begin{proposition}\label{intersection}
Let $P,Q\subset SM$ be smooth, oriented, Legendrian and relatively compact submanifolds of dimension $n-1$. Let $Y$ be a contact vector field on $SM$ transverse to $P$, and denote by $\phi_t$ the associated flow. Consider the submanifolds  $P_t=\phi_t(P)$ and $R=\bigcup_{0<t<r} P_t$ for small $r$. Choose on $R$ the orientation induced by $\phi\colon(0,r)\times P\rightarrow R$. Then
\[
 \int_G  \algint{R}{gQ}\  dg=(-1)^{n-1}\int_0^r\left[ \int_{P_t\times Q} \pi_1^*(\alpha(Y)) \cf \right]dt.
\]
\end{proposition}
\begin{proof}
Let $O=R\times Q\subset SM\times SM$. Then 
\begin{align}\notag
\algint{p^{-1}(O)}{q^{-1}(g)}&=  (-1)^{\dim G}\algint{q^{-1}(e)}{p^{-1}(R\times gQ)}\\
&\notag=(-1)^{\dim G} \algint{p(q^{-1}(e))}{(R\times gQ)}\\
&=(-1)^{n}\algint{R}{gQ} \label{gQR}.
\end{align}
In the first line we applied the orientation preserving diffeomorphism of $SM\times G$ given by $(\xi,h)\mapsto (g\xi, hg^{-1})$. The second line follows from \cite[(2.4)]{fu.intersection}. The third line follows since $p(q^{-1}(e))$ is the diagonal $\Delta$ in $SM\times SM$, with the orientations differing by a factor $(-1)^{\dim G}$, and it is a general fact that $\algint{\Delta}{(A\times B)}=(-1)^{\dim A}\algint{A}{B}$.
By the coarea formula (cf. e.g. \cite[(2.3)]{fu.intersection}) and the defining relation of push-forward (cf. \cite[(2.6)]{fu.intersection}),
\begin{equation}\label{coarea}
 \int_G {\algint{p^{-1}(O)}{q^{-1}(g)}}\  dg=\int_{p^{-1}O}q^*dg=\int_O p_*q^*dg.
\end{equation}

We will see in the proof of Proposition \ref{fiber_integration} that $p_*q^*dg={(\pi_2^*\alpha-\pi_1^*\alpha)}\wedge \cf$.
Putting  $\psi(t,\xi,\eta)=( \phi_t(\xi),\eta)$, we have $O=\psi([0,r]\times P\times Q)$, so
\begin{align}
\notag \int_O p_*q^*dg&=\int_O (\pi_2^*\alpha-\pi_1^*\alpha)\wedge\cf\\
&\notag=-\int_O\pi_1^*\alpha\wedge\cf\\
&\notag=-\int_{[0,r]\times P\times Q}\psi^*(\pi_1^*\alpha\wedge\cf) \\
&=-\int_0^r \left(\int_{ P_t\times Q}  \pi_1^*(\alpha(Y))\cf\right) dt\label{iotaT},
\end{align}
since $Q$ and all the $P_t$ are Legendrian.
The result follows from \eqref{gQR}, \eqref{coarea}, and \eqref{iotaT}.
\end{proof}

Note that Proposition \ref{intersection} applies also to the case where $P,Q$ are (Borel domains inside) normal cycles of smooth polyhedra. Indeed, for $A\in\mathcal P(M)$, the normal cycle $N(A)$ is a finite union of smooth relatively compact submanifolds of $SM$.

\begin{corollary}\label{meaning}
\begin{equation}\label{variation}
 (\delta\otimes\id) k(\chi)=(\id\otimes \glob) \cm.
\end{equation}
\end{corollary}
\begin{proof}
Given $B\in \mathcal P(M)$, we consider $\mu_B=k(\chi)(\,\cdot\,,B)$ and need to show $\delta \mu_B=b(\,\cdot\,,\,\cdot\,,B,B)$ as curvature measures. Given a regular domain $A$ and a vector field $X$, by \eqref{delta_description}
\[
\left.\frac{d}{dt}\right|_{t=0} \mu_B(A_t)=\int_{\partial A} \langle \nu_A(x),X(x)\rangle \delta \mu_B(A,dx),
\]
where  $A_t$ denotes the image of $A$ under the flow associated to $X$. On the other hand
\[
\int_{\partial A} \langle \nu_A(x),X(x)\rangle b(A,dx,B,B)=\int_{N(A)\times s(N(B))} \pi_1^*(\alpha(Y))\, \kappa,
\]
where $Y=X^c$ is the complete lift of $X$ to $SM$.
Hence,  we need to prove
\[
\left.\frac{d}{dt}\right|_{t=0}\mu_B(A_t)=\int_{N(A)\times s( N(B))} \pi_1^*( \alpha(Y))\,  \cf.
\] It is enough to consider the case when $\langle X,\xi\rangle >0$ as any vector field can be put as a difference of two such $X$. In this case, for almost every $g\in G$ (see \cite[Lemma 3.6]{fu.intersection} with another sign because we oriented $s(N(B))$ differently)
\[
 \chi(A_r\cap gB)-\chi(A\cap gB)=(-1)^{n-1}\algint{s(N(gB))}{\bigcup_{0<t<r} N(A_t)}.
\]
Hence, the result follows from Proposition \ref{intersection} with $P=N(A), Q=s(N(B))$. Indeed, in this case $\phi_t(N(A))=N(A_t)$.
\end{proof}

It is interesting to notice that Corollary \ref{meaning} can also be obtained from \cite[Theorem 6.3]{fu.intersection} by J. Fu. Let us note however that the factor $(-1)^n$ in the statement of that theorem should be removed, as it disappears in the last but one equality of the proof.  Taking this into account, Fu's theorem implies that $\delta \mu_B$ is given by 
\[
 \delta \mu_B(A,U)=\int_{N(A)\cap \pi^{-1}U}i_Ts^*\omega_B 
\]
where $\omega_B\in\Omega^{n}(SM)$ is such that
\[
\int_{SM} \beta\wedge \omega_B =\int_G\left(\int_{N(gB)} \beta\right)dg,\qquad \forall \beta\in\Omega_c^{n-1}(SM).
\]
Let $r_B\colon SM\times N(B)\rightarrow SM$ be the projection on the first factor, and let $j_B\colon SM\times N(B)\rightarrow SM\times SM$ be the inclusion. By \cite[(2.2)]{fu.intersection}, and taking care of the orientation conventions, we have
\begin{align*}\int_{SM}\beta\wedge(r_B)_*(j_B^*p_*q^*dg)&=\int_{SM\times N(B)}r_B^*\beta\wedge j_B^*p_*q^*dg\\
&=(-1)^{n}\int_{N(B)\times G} (r_B\circ p)^*\beta\wedge q^*dg.
\end{align*}Since
\begin{align*}\int_{N(B)\times G} (r_B\circ p)^*\beta\wedge q^*dg&=\int_G\left(\int_{N(gB)}\beta\right)dg,
\end{align*}
we see that $\omega_B$ can be explicitly presented as
$$\omega_B=(-1)^{n}(r_B)_*j_B^*p_*q^*dg\in \Omega^n(SM).$$
Since $s\colon N(A)\rightarrow s(N(A))$ reverses orientations, and $(s\otimes s)^*\kappa=(-1)^{n-1}\kappa$ by Proposition \ref{fiber_integration}, we have
\begin{align*}
 \delta\mu_B(A,U)&= (-1)^n \int_{N(A)\cap\pi^{-1}U} i_T(s^*(r_B)_*j_B^*p_*q^*dg)\\
&=(-1)^n\int_{N(A)\cap\pi^{-1}U} s^*(r_B)_*j_B^*\kappa\\
&=(-1)^{n-1}\int_{s(N(A)\cap\pi^{-1}U)\times N(B)}\kappa\\
&=\cm(A,U,B,B),
\end{align*}
which is the content of Corollary \ref{meaning}.  

A further consequence of Proposition \ref{intersection} is the following.

\begin{corollary}\label{firey}Let $(F,H\ltimes F)$ be  a Euclidean affine isotropic space with isotropy group $H$. Let $\cm_0\in\curv\otimes\curv$ be the contact measure on this space.
Let $K,L$ be convex bodies in  $F$ with smooth positively curved boundary and $U,V$ Borel sets in the unit sphere $S(F)$. Given $r>0$ let $m_r(K,U,L,V)$ be the Haar measure  of the set of $g\in H\ltimes F$ such that  
\[\min\{|y-x|\colon x\in K, y\in gL\}=|y_0-x_0|
\]
for some  $x_0\in K,y_0\in gL$ with $|y_0-x_0|<r$ and $(y_0-x_0)/|y_0-x_0|\in U\cap s(gV)$.Then
\[
\left.\frac{d}{dr}\right|_{r=0}m_r(K,U,L,V)=\cm_0(K,\gamma_K^{-1}U,L,\gamma_L^{-1}V),
\]
where $\gamma_K,\gamma_L$ are the Gauss maps of $\partial K,\partial L$ respectively.
\end{corollary}
\begin{proof} Let $K_t$ denote the parallel body to $K$ at distance $t$. Clearly $m_r(K,U,L,V)$ measures the set of $g$ such that $s\left(N(gL)\cap\pi_S^{-1}gV\right)$ intersects some $N(K_t)\cap \pi_S^{-1}U$ with $0<t<r$. Hence,
\[
m_r(K,U,L,V)=(-1)^{n-1}\int_{H\ltimes F}s \left(N(gL)\cap \pi_S^{-1}(gV)\right)\cdot \bigcup_{0<t<r}N(K_t)\cap \pi_S^{-1}(U) dg.
\]
 The result follows from Proposition \ref{intersection} with $Y=T$, the Reeb vector field, $P=N(K)\cap \pi_S^{-1}U$, and $Q=s(N(L)\cap \pi_S^{-1}V)$.
\end{proof}

The following proposition shows the relation between contact measures and the kinematic formula  for the surface area measure in affine isotropic spaces (cf. \cite{wannerer.aim}). This connection was originally established  by Schneider \cite{schneider} for $H=SO(n)$. 

\begin{proposition} Let $(F,H\ltimes F)$ be  a Euclidean affine $n$-dimensional isotropic space with isotropy group $H$. Let $\cm_0\in\curv\otimes\curv$ be the contact measure on this space. Consider the canonical injection $i\colon\mathrm{Area}\hookrightarrow\mathrm{Curv}$, and the involution $r\colon\mathrm{Area}\rightarrow\mathrm{Area}$ given by $r(\Phi)(K,U)=\Phi(-K,-U)$. Then 
\begin{equation}\label{connection}
 (i\otimes (i\circ r))A(S_{n-1})=\cm_0.
\end{equation}
Equivalently, given convex bodies $K,L\subset F$, and Borel domains $U,V\subset S(F)$ on the unit sphere
 \begin{equation}\label{akl}
  A(S_{n-1})(K,U,-L,-V)=
\int_{(N(K)\cap \pi_S^{-1}U)\times s(N(L)\cap \pi_S^{-1}V)} \cf,
 \end{equation} where $\pi_S$ is the projection of $SF=F\times S(F)$ to the second factor. 
\end{proposition}
\begin{proof} It was shown by Schneider (see \cite[(3.4)]{schneider}) that
\[
\left.\frac{d}{dr}\right|_{r=0}m_r(K,U,L,V)=A(S_{n-1})(K,U,-L,-V).
\]
In fact, only the case  $H=O(n)$ is considered in \cite{schneider} but the proof works verbatim for any $H$ acting transitively on the sphere. By Corollary \ref{firey}, this shows \eqref{akl} for $K,L$ with smooth positively curved boundary. For general convex bodies, the equation follows by the continuity of normal cycles (cf. e.g. \cite[Theorem 2.2.1]{fuBCN}).
\end{proof}

\section{Transfer principle}
Here we compute the form $\cf$ explicitly. The resulting expression shows that  contact measures of different spaces with the same isotropy group can be identified with each other through the transfer map. 
In turn, this reveals a connection  between  the principal kinematic formulas in a curved isotropic space, and the additive kinematic formula for the surface area measure in the corresponding affine space, which is our main result.

\begin{proposition}\label{fiber_integration}  
Given $\xi\in SM$, let $\xi,e_1,\ldots e_{n-1}$ be an orthonormal basis of $T_{\pi\xi}M$. 
Let $\theta_i=\langle d\pi,e_i\rangle$, and  $\varphi_i=\langle \nabla,e_i\rangle$, where $\nabla$ is the projection onto the second component of the decomposition $T_\xi SM=T_{\pi\xi}M\oplus \xi^\bot$ induced by the Levi-Civita connection.
Then, for any $h\in G$
 \begin{equation}\label{explicit}
  \cf_{(\xi,h^{-1}\xi)}=\frac{1}{n\omega_n} \int_K \bigwedge_{i=1}^{n-1}(\pi_1^*\theta_i-\pi_2^*(kh)^*\theta_i)\wedge\bigwedge_{j=1}^{n-1}(\pi_1^*\varphi_j-\pi_2^*(kh)^*\varphi_j)dk,
 \end{equation}
where  $K$ denotes the isotropy group of $\xi$ with its Haar probability measure $dk$, and $\pi_1,\pi_2\colon SM\times SM\to SM$ are the projections on each factor. 
\end{proposition}
\begin{proof} Let $\phi\colon G\to SM$ be given by $\phi(g)=g\xi$, and let $\gamma:U\to G$ be such that $\phi \circ \gamma=\id$ on an open neighborhood $U$ of $\xi$, and $\gamma(\xi)=e$. Then a local trivialization $U\times h^{-1}U\times K \cong E|_{U\times h^{-1}U}$ of $p$ is given by $(\zeta,\eta,k)\mapsto (\eta,\psi(\zeta,\eta,k))$ with $\psi(\zeta,\eta,k)=\gamma(\zeta)k\gamma(h\eta)^{-1}h$. Denoting by $\pi_1,\pi_2,\pi_K$ the projections onto the factors of $U\times h^{-1}U\times K$, and by $L_g,R_g$ the left and right translations on $G$, we have
\[
 (d\psi)_{(\xi,h^{-1}\xi,k)}=dR_{kh} (d\gamma)_\xi d\pi_1 - dL_kdR_{h} (d\gamma)_\xi (dh)_{h^{-1}\xi} d\pi_2+dR_{h}d\pi_K.
\]
The decomposition $T_eG\equiv T_\xi M\oplus \xi^\bot\oplus T_eK$ yields the following expression for the volume form of $G$ at $e$, which is fixed according to the orientation taken in $G$,
\[
 (dg)_{e}=\frac{-1}{n\omega_n}\phi^*(\langle d\pi,\xi\rangle)\wedge\bigwedge_{i=1}^{n-1}\phi^*\theta_i \wedge \bigwedge_{j=1}^{n-1}\phi^*\varphi_j\wedge dk.
\]
The negative sign is taken so that the volume form $dk$ induces the correct orientation on the fibers of $p$;  that is, so that the previous local trivializations preserve orientations. Hence, 
\[
(\psi^*dg)_{(\xi,h^{-1}\xi,k)}=(d\psi_{(\xi,h^{-1}\xi,k)})^*(dR_{(kh)^{-1}})^*(dg)_e=
\]\[=\frac{-1}{n\omega_n}(\pi_1^*\alpha-\pi_2^*\alpha)\wedge\bigwedge_{i=1}^{n-1}(\pi_1^*\theta_i-\pi_2^*(kh)^*\theta_i)\wedge\bigwedge_{j=1}^{n-1}(\pi_1^*\varphi_j-\pi_2^*(kh)^*\varphi_j)\wedge \pi_K^*dk
\]
and equation \eqref{explicit} follows. 
\end{proof}

\begin{corollary}\label{thm_smooth}Given smooth regular domains $A,B\subset M$, let  $S_A,S_B$ be the shape operators of $\partial A,\partial B$ corresponding to the exterior normal fields $\nu_A,\nu_B$. Fix $\xi\in SM$, and let $h=h(x,y)\in G$ be such that $h\cdot\nu_B(y)=s(\nu_A(x))$.  Then
\begin{equation}\label{smooth}
  \cm(A,U,B,V)=\frac{1}{n\omega_n}\int_{\partial A\cap U}\int_{\partial B\cap V} \int_{K_x}\det((S_A)_x+d(kh)\circ (S_B)_y\circ d(kh)^{-1})dk dydx,
 \end{equation}where $K_x$ stands for the isotropy group of $\nu_A(x)$,  endowed with the Haar probability measure $dk$, and $dx,dy$ are the volume elements of $\partial A,\partial B$ respectively.
\end{corollary}
The previous expression was already obtained by Teufel in \cite{teufel}. This confirms that  $\cm$ coincides, on smooth domains and up to normalization, with the contact measure introduced there.

\begin{proof} By \eqref{explicit}, the contact measure $\cm(A,\cdot,B,\cdot)$ is given by a differential form on $\partial A\times \partial B$ whose value at $(x,y)$ is
\begin{equation}\label{bigform}
(\nu_A\times s\nu_B)^*\cf=\frac{1}{n\omega_n}\int_{K_x}\bigwedge_i(\nu_A^*\theta_i-\nu_B^*(kh)^*\theta_i)\wedge \bigwedge_j(\nu_A^*\varphi_j+\nu_B^*(kh)^*\varphi_j) dk
\end{equation}
Putting  $\sigma_i=\langle \cdot, e_i\rangle$, we have 
\begin{align*}
 \nu_A^*\theta_i&=\sigma_i, &\nu_B^*(kh)^*\theta_i&=(kh)^*\sigma_i=\sum_j \sigma_i(k e_j)h^*\sigma_j,
\\
 \nu_A^*\varphi_i&=S_A^*\sigma_i, &\nu_B^*(kh)^*\varphi_i&=(k h S_B)^*\sigma_i=\sum_j \sigma_i(d(kh)\,S_B\,dh^{-1}\,e_j) h^*\sigma_j.&
\end{align*}
It follows that the integrand in \eqref{bigform} equals 
\begin{align*}
&\det \left(\begin{array}{cc}
 \id&-k\\S_A&d(kh)\,S_B\,dh^{-1}\end{array}\right)\bigwedge_i \sigma_i \wedge \bigwedge_j h^*\sigma_j
\\&=\det \left(\begin{array}{cc}
 \id&-\id\\S_A&d(kh)\,S_B\,d(kh)^{-1}\end{array}\right)dx\wedge dy\\
&=\det\left(S^A+d(kh) S^B d(kh)^{-1}\right)dx\wedge dy.
\end{align*}
\end{proof}

\begin{proposition}\label{trans} Let $(M,G)$ be a Riemannian isotropic space with isotropy group $H$, and let $(F,H\ltimes F)$ be the affine isotropic space with the same isotropy group. 
The respective contact measures $\cm$ and $\cm_0$ of $M$ and $F$ are identified with each other by
\[
\cm=(\tau\otimes\tau)\cm_0.
\]
\end{proposition}
\begin{proof}Recalling that $\tau$ is induced by \eqref{identif_forms}, the relation follows from \eqref{explicit}.
\end{proof}

We are finally able to prove Theorem \ref{mainthm} which we restate as follows.

\begin{corollary} Let $(M,G)$ be an isotropic space with isotropy group $H$, and let $(F,H\ltimes F)$ be the affine isotropic space with the same isotropy group. The principal kinematic formula $k(\chi)$ of $M$ is related to the additive kinematic formula $A(S_{n-1})$ of the surface area measure of $F$ through 
\begin{equation}\label{link}
  (\delta\otimes\id) k(\chi)= (\id\otimes \glob)\circ\left( (\tau\circ i)\otimes(\tau\circ i\circ r)\right) A(S_{n-1}).
\end{equation}
\end{corollary}
\begin{proof}This follows from \eqref{variation}, \eqref{connection} and Proposition \ref{trans}.
\end{proof}

Recalling the classification of isotropic spaces (cf. \cite[\S 4D]{tits}), one realizes that in most cases $r$ is superfluous in \eqref{link}. Indeed, except for the exceptional spaces $(\mathrm S^6,G_2), (\mathbb R\mathrm P^6,G_2), (\mathrm S^7, Spin(7))$ and $(\mathbb R\mathrm P^7,Spin(7))$, all the non-flat isotropic spaces are symmetric; i.e. the reflection about any point belongs to the isotropy group, in which case $r$ can be omitted in equation \eqref{link}. 

\section{Hermitian contact measures}\label{Hermitian}
As an application of Theorem \ref{mainthm} we will use \eqref{link} and the results of \cite{bfs} to determine the contact measure $\cm$, and equivalently the kinematic formula for the surface area measure $A(S_{2n-1})$, in Hermitian space $\mathbb C^n$. In fact, an algorithm to find such formulas was already given in \cite{wannerer.aim}, but our expressions are a bit more explicit. On the other hand, the results of \cite{wannerer.aim} include many invariant area measures that seem out of reach  by our approach.

Let $\mathrm{Area}^{U(n)}$ denote the space of $U(n)$-invariant area measures in $\mathbb C^n$. A basis of the vector space  $\mathrm{Area}^{U(n)}$ was  determined in \cite{wannerer.jdg} and consists of two distinct families $\{\Delta_{k,q}\}\cup\{N_{k,q}\}$ with the following range of indices: 
\begin{align*}
\Delta_{k,q},&\qquad 0,k-n \leq q\leq k/2\leq n\\
N_{k,q},&\qquad  0,k-n < q< k/2\leq n.
\end{align*}
We refer to \cite{wannerer.jdg} for the definition of these area measures. Their main property is the following. Let $\glob: \areaun\rightarrow \valun$ be the globalization map given by $\glob(\Theta)(A)=\Theta(A,S^{2n-1})$. Then 
\[
\glob(N_{k,q})=0,\qquad
\glob(\Delta_{k,q})=\mu_{k,q},
\]
where $\mu_{k,q}$ denote the so-called {\em Hermitian intrinsic volumes}, a basis of $\valun$ introduced in \cite{hig}. 

Another important family in $\areaun$ is 
\begin{align}
B_{k,q}&=\Delta_{k,q} + N_{k,q}, \qquad 0,k-n < q<k/2\leq n\\
B_{k,q}&=\Delta_{k,q},\qquad \max\{0,k-n\}=q\leq k/2\leq n.
\end{align}
For $k\neq 2q$ this coincides with  the notation used in \cite{bfs}, and \cite{wannerer.jdg}, but it must be noticed that $B_{2q,q}$ was instead denoted by $\Gamma_{2q,q}$ there. This little departure will  simplify some expressions.

It will be useful to consider the map $\ell_B:\valun\rightarrow \areaun$ given by $\ell_B(\mu_{k,q})=B_{k,q}$. Note that $\glob\circ\ell_B=\id$. Also, we will decompose the first variation operator $\delta_A:\valun\rightarrow \areaun$ introduced in \cite{wannerer.jdg}, as $\delta_A=\delta_B+\delta_N$ with $\delta_B=\ell_B\circ \glob\circ\delta$. Explicitly (cf. \cite[Proposition 4.6]{hig}),
\[
\delta_B(\mu_{k,q})=c_{n,k,q}\left(\frac{2(k-2q)}{c_{n,k-1,q}} B_{k-1,q}+\frac{q}{c_{n,k-1,q-1}}B_{k-1,q-1}\right)
\]
\[
\delta_N(\mu_{k,q})=c_{n,k,q}\left(\frac{(k-2q)^2(2n-k+1)}{c_{n,k-1,q}(n-k+q+1)}N_{k-1,q}-\frac{q(2n-k+1)}{c_{n,k-1,q-1}} N_{k-1,q-1}\right),
\]
where 
\[c_{n,k,q}=\frac{1}{q!(n-k+q)!(k-2q)!\omega_{2n-k}}.\]

With all this notation at hand, the kinematic formula for the surface area measure in $\Cn$ is explicitly
 given as follows. Recall that the principal kinematic formula $k(\chi)$ of $\Cn$ was computed in \cite{hig}.
\begin{theorem}\label{app}
The additive kinematic formula $A(S_{2n-1})$ for the surface area measure $S_{2n-1}$ in $(\mathbb C^n, U(n)\ltimes \mathbb C^n)$ is obtained from the principal kinematic formula $k(\chi)$ of this space by
\[
A(S_{2n-1})=(\delta_A\otimes \ell_B+\ell_B\otimes \delta_N) k(\chi).
\]
\end{theorem}
The idea of  the proof is to use \eqref{link}, and the principal kinematic formula $k_\lambda(\chi)$ of $\CPnlam$, the complex space form of constant holomorphic curvature $4\lambda$, computed in \cite{bfs}. Indeed, the globalization map from $\curvun$ to the space $\V_n^\lambda$ of isometry invariant valuations in $\CPnlam$ has a different kernel for each value of $\lambda$. In fact, these kernels are in direct sum, so the knowledge of $k_\lambda(\chi)$ for different $\lambda$ is enough to recover $A(S_{2n-1})$.

The proof of Theorem \ref{app} will be easy once we have established the following.
\begin{proposition}\label{noNN}
Let $\mathcal B=\mathrm{span}\{B_{kq}\}$ and $\mathcal N=\mathrm{span}\{N_{kq}\}$. Then
 \[
A(S_{2n-1})\in(\mathcal B\otimes\mathcal B)\oplus (\mathcal B\otimes\mathcal N)\oplus (\mathcal N\otimes\mathcal B)  
 \]
\end{proposition}
\begin{proof} 
Since clearly $A(S_{2n-1})\in \Sym^2(\areaun)$, let us write 
\[
 A(S_{2n-1})=\sum_{k,q,p} a_{kqp} B_{k,q}\otimes B_{k',p}+\sum_{k,q,p} b_{kqp} B_{k,q}\odot N_{k',p}+\sum_{k,q,p} c_{kqp} N_{k,q}\otimes N_{k',p},
\]
where $\odot$ denotes  the symmetric product and $k'=2n-k-1$. We need to show that all $c_{kqp}=0$.

Let  $\tau_\lambda$ be the transfer map from $\curvun$ to the space $\mathcal C_n^\lambda$ of invariant curvature measures in $\CPnlam$, and let 
\[g_\lambda\colon\areaun\stackrel{i}{\longrightarrow}\curvun\stackrel{\tau_\lambda}{\longrightarrow}\mathcal C_n^\lambda\stackrel{\glob}{\longrightarrow} \V_n^\lambda\]
that is $g_\lambda(\Phi)(P)=\tau_\lambda(i(\Phi))(P,P)$. In terms of the basis $\{\mu_{k,q}^\lambda\}$ of $\V_n^\lambda$ used in \cite{bfs}, the map $g_\lambda$ is  given by (cf. \cite[Lemma 3.3, Definition 3.7, and Lemma 3.9]{bfs})
\begin{align*}
 g_\lambda(N_{k,q})&=-\lambda\frac{q+1}{\pi}\mu_{k+2,q+1}^\lambda,\\
g_\lambda(B_{k,q})&=\mu_{k,q}^\lambda,\qquad k\neq 2q\\
g_\lambda(B_{2q,q})&=\mu_{2q,q}^\lambda-\lambda\frac{q+1}{\pi}\mu_{2q+2,q+1}^\lambda.
\end{align*}
Hence, modulo $\mathcal B\otimes\V_n^\lambda$,
\begin{align}\notag
 (\id\otimes g_\lambda) A(S_{2n-1})&\equiv 
\sum_{k,q,p} b_{kqp} N_{k',p}\otimes\mu^\lambda_{k,q}-\frac{\lambda}{\pi}\sum_{q,p}(q+1) b_{2q,q,p} N_{2n-2q-1,p}\otimes \mu_{2q+2,q+1}^\lambda\\\label{AS}
& -\frac{\lambda}{\pi}\sum_{k,q,p} c_{kqp}(p+1) N_{k,q}\otimes \mu^\lambda_{k'+2,p+1}.
\end{align}

For $\lambda>0$ and $0\leq q\leq n$, let $P_\lambda$ be a totally geodesic complex subspace of $\CPnlam$ of complex dimension $q$.
Since (see the proof of Propostion 5.2 in \cite{bfs})
\[
\mu_{k,p}^\lambda(P_\lambda)=\delta_{k,2q}\delta_{p,q}\frac{\pi^q}{\lambda^qq!},
\]
one has
\begin{equation}\label{lim}
\lim_{\lambda\to 0}\lambda^q(\id\otimes g_\lambda) A(S_{2n-1})(\cdot,P_\lambda)\equiv\frac{\pi^{q}}{q!}\sum_p b_{2q,q,p}N_{2n-2q-1,p} 
\end{equation}
modulo $\mathcal B$. 

We denote by $\delta_\lambda:\V_n^\lambda\rightarrow \curvun$ the first variation operator in $\CPnlam$ composed with $\tau_\lambda^{-1}$. 
By \eqref{link}, since $r=\id$ on $\areaun$ as $-\id\in\mathrm U(n)$,
\begin{equation}\label{d_lambda_k}
 (\delta_\lambda\otimes\id)k_\lambda(\chi)=(i\otimes g_\lambda) A(S_{2n-1}).
\end{equation}
However, by  \cite[Corollary 4.5]{ags}, we know that
\begin{equation}\label{lemags}
  (\delta_\lambda\otimes\id)k_\lambda(\chi)(\cdot,P_\lambda)\in\mathcal B.
\end{equation}
Comparing \eqref{d_lambda_k},\eqref{lim} and \eqref{lemags} shows that $b_{2q,q,p}=0$ for all $q,p$.

On the other hand, the principal kinematic formula  in $\CPnlam$ has the form 
\[
k_\lambda(\chi)=\sum d_{k,q,p}\mu_{k,q}^\lambda\otimes\mu_{k',p}^\lambda,
\]
for certain constants $d_{k,q,p}$  independent of $\lambda$ (cf. Theorem 3.19 \cite{bfs}). Also $\delta_\lambda$ is known explicitly and has the form (cf. \cite[Proposition 3.7]{ags})
\begin{align}\label{d_lambda}
\delta_\lambda(\mu_{k,q}^\lambda)&=\delta_0(\mu_{k,q})+\lambda e_{k,q} i(B_{k+1,q})+\lambda f_{k,q}i(B_{k+1,q+1}).\notag
\end{align} for some $e_{k,q},f_{k,q}$. Together with \eqref{d_lambda_k}, the last two relations give
\begin{equation}\label{dk}
(i \otimes g_\lambda) A(S_{2n-1})=\sum d_{k,q,p}\delta_0(\mu_{k,q})\otimes\mu_{k',p}^\lambda+\lambda\sum d'_{k,q,p}i(B_{k+1,q})\otimes\mu_{k',p}^\lambda
\end{equation}
for certain constants $d_{k,q,p}'$.

By comparing the coefficients of $\lambda$ in \eqref{AS} and \eqref{dk}, and recalling that $b_{2q,q,p}=0$, we see that all $c_{n,k,q}=0$, as claimed. 
\end{proof}

\begin{proof}[Proof of Theorem \ref{app}]Let us denote $D=A(S_{2n-1})-(\delta_A\otimes \ell_B+\ell_B\otimes \delta_N) k(\chi)$. We noted before that $A(S_{2n-1})\in\Sym^2\areaun$. Also, by \eqref{connection}
\[
 (\delta_B\otimes\ell_B)k(\chi)=(\ell_B\otimes\ell_B)\circ (\glob\circ\delta\otimes\id) k(\chi)=(\ell_B\circ\glob\circ i)\otimes (\ell_B\circ\glob\circ i)A(S_{2n-1})
\] is symmetric.
This shows that $D\in \Sym^2\areaun$. Note also that, by \eqref{link}
\[
 (\id\otimes\glob)D=0.
\]
Hence, 
\[
 D\in \ker(\id\otimes\glob)\cap \Sym^2\areaun\subset \mathcal N\otimes \mathcal N.
\]
But $D\in(\mathcal B\otimes\mathcal B)\oplus (\mathcal B\otimes\mathcal N)\oplus (\mathcal N\otimes\mathcal B)$  by Proposition \ref{noNN}, so $D=0$.
\end{proof}

\end{document}